\documentclass{article}

\usepackage{amssymb,amsfonts,amsmath,latexsym,color,amscd} 
\usepackage{url}

\usepackage[dvips]{psfrag,graphicx} 

\usepackage{cases}

\newtheorem{theorem}{Theorem}[section]
\newtheorem{lemma}[theorem]{Lemma}
\newtheorem{corollary}[theorem]{Corollary}

\newtheorem{definition}[theorem]{Definition}

\newenvironment{proof}{{\par\addvspace{0.1cm}\noindent \bf Proof. }}{\hfill$\Box$\par\medskip} 

\newtheorem{conjecture}[theorem]{Conjecture}
\newtheorem{remark}[theorem]{Remark}

\numberwithin{equation}{section}

\def\a{\alpha}
\def\e{\varepsilon}
\def\R{\Re\mathfrak{e} \,}
\def\vect#1{\mbox{\boldmath $#1$}} 
\def\RR{\mathbb{R}}

\def\NN{\mathbb N}
\def\ZZ{\mathbb Z}
\def\Pf{\mbox{\rm Pf.}}

\def\quack{\hspace{0.4cm}}

\newcommand\bcdot{\boldsymbol{\cdot}}
\def\jb#1{\textcolor[named]{Black}{{#1}}} 
\def\j#1{\textcolor[named]{Black}{ {#1}}} 

\begin{document}

\title{Self-repulsiveness of energies for closed submanifolds}

\author{Jun O'Hara\footnote{Supported by JSPS KAKENHI Grant Number 19K03462.}}
%
%
\maketitle

\begin{abstract}
We show that the regularized Riesz $\a$-energy for \j{smooth} closed submanifolds $M$ in $\RR^n$ blows up as $M$ degenerates to have double points if $\a\le-2\dim M$. 
This gives theoretical foundation of numerical experiments to evolve surfaces to decrease the energy which have been carried out since \j{the} 90's. 
\end{abstract}



\section{Introduction}

The {\it energy} of a knot 
\begin{equation}\label{def_energy_knot}
E(K)=\int_K\left[\lim_{\e\to 0^+}\left(\int_{K, |x-y|\ge\e}\frac{dy}{|x-y|^2}-\frac{2}\e\right)\right]dx 
\end{equation}
was introduced in \cite{O1}\footnote{
To be precise, in \cite{O1} we defined the energy by 
$\displaystyle E_O(K)=\int_K\left[\lim_{\e\to 0^+}\left(\int_{K, d_K(x,y)\ge\e}\frac{dy}{|x-y|^2}-\frac{2}\e\right)\right]dx,$ where $d_K(x,y)$ is the arc-length between $x$ and $y$ along the knot. As was pointed out in \cite{Obook} Remark 2.2.1, $E_O(K)$ and $E(K)$ coincide. 
}
motivated by a problem proposed by Fukuhara and Sakuma 
to ask for a functional on the space of knots which produces an ``optimal knot'' for each isotopy class (knot type) as an embedding that gives the minimum value within the isotopy class. 
For this purpose we like to evolve a knot along the gradient of the functional. 
To prevent self-crossing during the evolution process since it might change the knot type, we impose our functional to blow up if a knot degenerates 
to a singular knot with double points. 
This property is called {\it self-repulsiveness} (\cite{O1}) and has been studied in \cite{BR, 
Di-Er-Re98, LO, O92, R, SSvdM2013} etc. We remark that this property is called {\it charge} in \cite{Di-Er-Re98}.

It turns out that the energy $E$ is invariant under M\"obius transformations $T$ of $\RR^3\cup\{\infty\}$ that do not open the knot, i.e. $E(T(K))=E(K)$ as far as $T(K)$ is compact. 
The {\it M\"obius invariance} was proved by Freedman, He, and Wang \cite{FHW}, and $E+4$ is called the {\it M\"obius energy} of knots. 
They used M\"obius invariance to show that in each prime knot type there is a knot that minimizes the energy $E$ in the knot type. 
Other M\"obius invariant knot energies were studied in \cite{KS, LO}. 

Various generalizations of $E$ have been studied. 
Note that the energy is obtained by the regularization (renormalization) of a divergent integral $\iint_{K\times K}|x-y|^{-2}dxdy$. The technique used in \eqref{def_energy_knot} is called {\it Hadamard's regularization}, 
which was applied to surfaces by Auckly and Sadun \cite{AS}. 
Doyle and Schramm found the {\it cosine formula} of $E(K)$ (reported in \cite{AS,KS}), which was generalized for surfaces and closed submanifolds by Kusner and Sullivan \cite{KS}. 
In both cases, the generalization was carried out keeping the M\"obius invariance property. 
On the other hand, Brylinski \cite{B} reformulated the energy $E$ using analytic continuation (meromorphic regularization), which was generalized to higher dimension by Fuller and Vemuri \cite{FV} and also to compact bodies by Solanes and the author \cite{OS2}. 

In higher dimensions, while the M\"obius invariance has been studied in \cite{AS,KS,OS2}, the self-repulsiveness of the energies mentioned above does not seem to have been studied as far as the author knows, even though numerical experiments to evolve surfaces so as to decrease the energy have been carried out since the 90's, for example \cite{Bra1}. 
We remark that the self-repulsiveness of the integral Menger curvature and related global curvature energies has been studied in \cite{BK, K, KSvdM13, KSvdM18, SvdM, SvdM13}. 
In this paper we show the self-repulsiveness of the generalized energies for surfaces and higher dimensional closed submanifolds in $\RR^n$ with the uniform control on $C^\a$ topology $(\a\ge3)$ when the energies are either scale invariant or ``stronger'' (subcritical) using results of recent studies on regularized Riesz energies \cite{O2019,OS2}. 
Thus we give theoretical foundation of numerical experiments. 
We also discuss the behavior or the energy when surfaces have double points. 

\medskip
Acknowledgement: The author would like to thank Professor Paul M.N. Feehan for the inquiry of this problem. He would also like to thank the anonymous reviewer for the very careful reading and many insightful suggestions.

\section{Preparation from regularized Riesz energies}

Let us explain a general scheme of Hadamard regularization. 
Suppose $\int_X\omega$ blows up on $\Delta\subset X$. 
Restrict the integration to the complement of an $\e$-neighborhood of $\Delta$, 
expand the result in a Laurent series in $\e$ (possibly with a $\log$ term and terms with non-integral powers), and finally take the constant term in the series. 
The constant is called {\it Hadamard's finite part} of the integral, denoted by $\Pf\int_X\omega$. 
It can be considered as a generalization of Cauchy's principal value. 

For an $m$-dimensional closed submanifold $M$ in $\RR^n$ and for any real number $\a$ we define the {\it regularized $\a$-energy} $E_\a(M)$ by 
\begin{equation}\label{def_of_E_alpha}
E_\a(M)=\int_M\left(\Pf\int_M {|x-y|}^\a dy\right)dx 
\end{equation}
(\cite{OS2}). 
Note that we do not need any regularization when $\a>-m$. 
In Proposition 3.7 of \cite{OS2} we showed that $E_\a(M)$ is the same as the quantity obtained by applying regularization by analytic continuation to a function 
\begin{equation}\label{def_of_I_M_z}
z\mapsto\int_M \left(\int_{M}{|x-y|}^z dy\right)dx, 
\end{equation}
where $z$ is a complex variable. 
It is well-defined and 
holomorphic when $\R z>-m$. We can expand the domain to the whole complex plane by analytic continuation. Then we obtain a meromorphic function, which we denote by $B_M(z)$, that has only (possible) simple poles at $z=-m, -m-2, -m-4, \dots$. We call it {\it Brylinski's beta function of $M$}. We remark that some of the residues might be zero. Then the relation between the Hadamard regularization and the regularization by analytic continuation is given as follows. 
\[
E_\a(M)=\left\{\begin{array}{ll}
B_M(\a) & \qquad \mbox{if $\a$ is not a pole of $B_M(z)$},\\[1mm]
\displaystyle \lim_{w\to \a}\biggl(B_M(w)-\frac{{\rm Res}(B_M;\a)}{w-\a}\biggr) & \qquad \mbox{if $\a$ is a pole of $B_M(z)$}.
\end{array}
\right.
\]
This can be proved by using the coarea formula to reduce the integrals in the right hand sides of \eqref{def_of_E_alpha} and \eqref{def_of_I_M_z} to an integral of the form which has been studied in the theory of generalized functions (cf. for example, \cite{GS}). We will use a similar argument in the proof of Lemma \ref{lemma1} later. 

In what follows we fix the dimensions of submanifolds, $m$, and of the ambient space, $n$. 

\begin{definition}\label{def_M(e,b,V)} \rm 
Let $\e_0, b, V$ be positive numbers. 
Let $\mathcal{M}(k,\e_0,b,V)$ $(k\ge3)$ be the set of closed immersed submanifolds $M'$ in $\RR^n$ with volume not greater than $V$ such that for each $M'$ there is a closed manifold $M$ of class $C^{k}$ and an immersion $f\colon M\to M'=f(M)$ of class $C^{k}$ that satisfy the following conditions. 
\begin{enumerate}
\item For any point $x$ of $M'$, $f^{-1}(B_{\e_0}(x)\cap M')$ is a disjoint union of $W_1, \dots, W_\ell$. 
\item Each $W_i$ is $C^{k}$-diffeomorphic to an $m$-ball. 
\item The restriction of $f$ to $W_i$ is a $C^{k}$-diffeomorphism to $W_i'=f(W_i)\subset M'$ for each $i$. 
\item Each $W_i'$ can be expressed as a graph of a function of class $C^{k}$ 
\[
h_i\colon U_i\to {\big(T_{x}W_i'\,\big)}^\perp\cong\RR^{n-m}, \quack U_i\subset T_{x}W_i'
\]
that satisfies $\left|\partial^\mu h_i\right|\le b$ on $U_i$ for any orthogonal axes for $T_{x}W_i'$ and for any multi-index $\mu$ with $0\le |\mu|\le k$. 
\end{enumerate}

\end{definition}

Let $M'\in \mathcal{M}(k,\e_0,b,V)$, $x\in M'$ and $W_i'$ be as in the above definition. 
Put $\psi_{x,i}(t)$ $(1\le i\le l, 0\le t\le \e_0)$ to be the volume of $B_t(x)\cap W_i'$. 
Proposition 2.3 of \cite{O2019}\footnote{We remark that it is a kind of $C^k$ analogues of the argument in the proof of Theorem 3.3 of \cite{FV} and Proposition 3.1 and Corollary 3.2 of \cite{OS2}.}
implies that each $\psi_{x,i}'(t)$ can be expressed by 
\[
\psi_{x,i}'(t)=t^{m-1} \, \overline{\varphi}_{x,i}(t), 
\]
for some $\overline{\varphi}_{x,i}(t)$ of class $C^{k-2}$ that satisfies\footnote{We use the notation $\overline{\varphi}$ to be consistent with the notation in \cite{O2019} to avoid confusion with $\varphi$ used in \cite{OS2}.} 
\begin{equation}\label{overline_varphi_coeff}
\overline{\varphi}_{x,i}(0)=\sigma_{m-1}, \hspace{0.5cm}
\overline{\varphi}_{x,i}^{\,(2j-1)}(0)=0
\>\>\>(1\le 2j-1 \le k-2), \nonumber
\end{equation}
where $\sigma_{m-1}$ is the volume of the unit $(m-1)$-dimensional sphere. 

We show that $\overline{\varphi}_{x,i}^{\,(j)}(t)$ $(0\le j\le k-2)$ approaches $0$ if $\partial^\mu h_i$ $(0\le |\mu|\le j+1)$ all tend to $0$. For this purpose we introduce some results from \cite{O2019}. \\
\begin{lemma}\label{lemma_regularity}\rm{(Lemma 2.2 of \cite{O2019})}
Suppose $g$ is a function of class $C^{k+1}$ $(k\ge1)$ from a neighbourhood of $0$ to $\RR$. 
If $g$ satisfies $g(0)=g'(0)=0$ then $\bar g$ defined by 
\[
\bar g(t)=\left\{\begin{array}{ll}
\displaystyle \frac{g(t)}t & \hspace{0.4cm} t\ne0, \\[3mm]
0 & \hspace{0.4cm} t=0
\end{array}
\right.
\]
is of class $C^k$, and 
\begin{equation}\label{derivative_bar_g}
{\bar g}^{(j)}(0)=\frac{g^{(j+1)}(0)}{j+1} \hspace{0.4cm} (0\le j\le k).
\end{equation}
\end{lemma}

\begin{lemma}\label{lemma_overline_b}
Put $\overline\a=\max\{\lfloor -\a \rfloor-m,0\}$ for $\a\in\RR$. 
There are positive numbers $\tilde \e_0$ and $\tilde b$ such that if $0<\e_0<\tilde \e_0$ and $0<b<\tilde b$ then 
there is a positive constant $\overline b=\overline b(\a,\e_0,b)$ 
such that 
$\big|\overline{\varphi}_{x,i}^{\,(j)}(t)\big|\le \overline b$ for any $M'\in \mathcal{M}(\overline\a+3,\e_0,b,V)$, $x\in M'$, $i=1,\dots,l$, $1\le j\le \overline\a+1$ and $t\in[0,\e_0]$. 
\end{lemma}

\begin{proof}
We use the argument in the proof of Proposition 2.3 of \cite{O2019}. 
Let us fix $i$. 
Define $H\colon U_i\to\RR^n$ by $H(y)=(y,h_i(y))$. 
Let $J(y)$ be the Jacobian of $H$: 
\begin{equation}\label{Jacbian}
J(y)=\sqrt{\det\left\langle\frac{\partial H}{\partial x_p}(y), \, \frac{\partial H}{\partial x_q}(y)\right\rangle}
=\sqrt{\det\left(\left\langle\frac{\partial h_i}{\partial x_p}(y), \, \frac{\partial h_i}{\partial x_q}(y)\right\rangle+\delta_{pq}\right)}
\>,\hspace{0.6cm}(p,q=1,\dots,m).
\end{equation}
For a unit vector $v\in S^{m-1}$, define $t(v,\bcdot\,)$ on a neighbourhood of $0$ by  
\[
t(v,\xi)=\left\{
\begin{array}{ll}
\displaystyle 
\mbox{sgn}(\xi)|H(\xi v)|=\xi\,\sqrt{1+\left|\frac{h_i(\xi v)}\xi\right|^2} 
&\hspace{0.4cm}(\xi\ne0), \\[3mm]
0&\hspace{0.4cm}(\xi=0),
\end{array}
\right.
\]
where $\mbox{sgn}(\xi)$ is the signature of $\xi$. 
Let $\xi(v,\bcdot\,)$ be the inverse function of $t(v,\bcdot\,)$;
\[
\xi(v,t)=\left\{
\begin{array}{ll}
\displaystyle \xi\ge0 \>\mbox{ such that }\>{\xi}^2+{|h_i(\xi v)|}^2=t^2 & \hspace{0.4cm} (t\ge0), \\[2mm]
\displaystyle \xi<0 \>\mbox{ such that }\>{\xi}^2+{|h_i(\xi v)|}^2=t^2 & \hspace{0.4cm} (t<0). \\[2mm]
\end{array}
\right.
\]
Since $h_i$ is of class $C^{\overline\a+3}$, Lemma \ref{lemma_regularity} implies that $t(v,\bcdot\,)$ is of class $C^{\overline\a+2}$, and therefore $\xi(v,\bcdot\,)$ is of class $C^{\overline\a+2}$. 
We have $\xi'(v,0)=1$ and $\xi(-v,t)=-\xi(v,-t)$. \\
\indent
By substituting $g(t)=\xi(v,t)-t$ in Lemma \ref{lemma_regularity} we obtain 
\[
\left. \frac{\xi(v,t)}{t}\right|_{t=0}=1. 
\]
The proof of Proposition 2.3 of \cite{O2019} shows 
\begin{equation}\label{varphi_bar}
\overline{\varphi}_{x,i}(t)=\int_{S^{m-1}} J(\xi(v,t)\, v)\left(\frac{\xi(v,t)}{t}\right)^{m-1}\,\frac{\partial\xi}{\partial t}(v,t)\, dv.
\end{equation}
Since $\xi(v,\bcdot\,)$ is of class $C^{\overline\a+2}$, both 
$\frac{\xi(v,t)}{t}$ and $\frac{\partial\xi}{\partial t}(v,t)$ are of class $C^{\overline\a+1}$ as functions of $t$ by Lemma \ref{lemma_regularity}, which implies that $\overline{\varphi}_{M,x}(t)$ is of class $C^{\overline\a+1}$. 
Since $\frac{\partial t}{\partial\xi}(v,0)=1$, on a neighbourhood of $0$ there holds $\frac{\partial t}{\partial\xi}(v,\xi)>0$ and 
\begin{equation}\label{dxidt}
\frac{\partial\xi}{\partial t}(v,t)=\left[\frac{\partial t}{\partial\xi}\big(v,\xi(v,t)\big)\right]^{-1},
\end{equation}
where 
\begin{equation}\label{dtdxi}
\frac{\partial t}{\partial\xi}(v,\xi)
=\frac{\displaystyle \xi+\sum_{p=1}^m v_p\left\langle\frac{\partial h_i}{\partial x_p} (\xi v), \,h_i(\xi v)\right\rangle}
{\displaystyle \xi\,\sqrt{1+\left|\frac{h_i(\xi v)}\xi\right|^2}} 
\quad (\xi\ne0), \qquad \frac{\partial t}{\partial\xi}(v,0)=1. 
\end{equation}
\\
\indent
By \eqref{Jacbian}, \eqref{dxidt}, \eqref{dtdxi} and \eqref{varphi_bar}, when $j$ satisfies $0\le j\le \overline\a+1$, $\overline{\varphi}_{x,i}^{\,(j)}(t)$ depends continuously on $\partial^\mu h_i$ $(0\le |\mu|\le j+1)$. 
Since $\overline{\varphi}_{x,i}^{\,(j)}(t)\equiv0$ for $j\ge1$ if $h_i\equiv0$, the conclusion follows. 
\end{proof}

\begin{remark}\rm 
Since $M$ is compact, the above proof in fact implies a stronger statement: \\
Let $\overline\a=\max\{\lfloor -\a \rfloor-m,0\}$ and $V>0$. 
For any positive number $\hat b$ there are positive numbers $\e_0'$ and $b'$ such that for any $M'\in \mathcal{M}(\overline\a+3,\e_0',b',V)$, $x\in M'$, $i=1,\dots,l$, $1\le j\le \overline\a+1$ and $t\in[0,\e_0']$ we have $\big|\overline{\varphi}_{x,i}^{\,(j)}(t)\big|< \hat b$. 
\end{remark}
\begin{lemma}\label{lemma1}
For any real number $\a$ and for any positive numbers $\e_0,b,V$ 
there is a positive number $b^0=b^0(\a,\e_0,b)$ such that 
for any $M'$ in $\mathcal{M}(\overline\alpha+3,\e_0,b,V)$, for any point $x$ in $M'$, for any $W_i'$, where $W_i'$ is as in Definition \ref{def_M(e,b,V)}, 
we have
\[
\left|\Pf \int_{B_{\e_0}(x)\cap W_i'}{|x-y|}^\a dy \right|\le b^0.
\]
\end{lemma}

\begin{proof}
We may assume that $\a\le-m$ and hence $\overline\alpha=\lfloor -\a \rfloor-m$ since if $\a>-m$ then the conclusion is trivial as the integral converges without regularization. 

We use the convention $\xi^0/0=\log \xi$ for $\xi>0$ in what follows to make the formulae simpler. 

The argument in what follows is somehow parallel to that in \cite{OS2} from Proposition 3.1 to Proposition 3.3, 
where we use the coarea formula to reduce the integrand of the energy (i.e the inner integral of \eqref{def_of_E_alpha}) 
into an integral of the form $\int_0^d t^w\varphi(t)\,dt$ for some $d>0$ and some smooth (or with suitable regularity) function $\varphi$ which has been studied in the theory of generalized functions (cf. for example, \cite{GS}). 

For $0<\e\le\e_0$ we have  
\begin{equation}\label{f_l22-1}
\int_{(B_{\e_0}(x)\cap W_i')\setminus B_{\e}(x)}{|x-y|}^\a dy 
=\int_\e^{\e_0}t^{\a+m-1}\,\overline\varphi_{x,i}(t)\,dt. 
\end{equation}
Noting that $\overline\varphi_{x,i}(t)$ is of class $\overline\a+1$ we see 
\begin{eqnarray}
\displaystyle \int_\e^{\e_0}t^{\a+m-1}\,\overline\varphi_{x,i}(t)\,dt
&=&\displaystyle \int_\e^{\e_0}t^{\a+m-1}\left(\overline\varphi_{x,i}(t)-\sum_{j=0}^{\overline\a}\frac{\overline\varphi_{x,i}^{(j)}(0)}{j!}\,t^j\right)dt \label{f_l22-3}\\[2mm]
&&\displaystyle  +\sum_{j=0}^{\overline\a}\frac{\overline\varphi_{x,i}^{(j)}(0)}{(\a+m+j)j!}\,\left({\e_0}^{\a+m+j}-{\e}^{\a+m+j}\right). \label{f_l22-4}
\end{eqnarray}
We agree that if $\a\in\ZZ$ the last term of \eqref{f_l22-4} when $j=\overline\a$, where $\overline\a=-\a-m$, is meant to be 
\[
\frac{\overline\varphi_{x,i}^{(\overline\a)}(0)}{\overline\a\, !}\,\left(\log{\e_0}-\log{\e}\right)
\]
by our convention $\xi^0/0=\log\xi$ for $\xi>0$. 

There is a constant $c\in(0,1)$ such that 
\begin{equation}\label{f_l22-2}
\overline\varphi_{x,i}(t)=\sum_{j=0}^{\overline\a}\frac{\overline\varphi_{x,i}^{(j)}(0)}{j!}\,t^j
+\frac{\overline\varphi_{x,i}^{(\overline\a+1)}(ct)}{(\overline\a+1)!}\,t^{\overline\a+1} 
\end{equation}
so that Lemma \ref{lemma_overline_b} implies that 
\begin{equation}\label{f25}
\left| t^{\a+m-1}\left(\overline\varphi_{x,i}(t)-\sum_{j=0}^{\overline\a}\frac{\overline\varphi_{x,i}^{(j)}(0)}{j!}\,t^j\right) \right|\le \frac{\overline b}{(\overline\a+1)!}\,t^{\a+m+\overline\a},
\end{equation}
where $\overline b=\overline b(\a,\e_0,b)$ is a positive constant given by Lemma \ref{lemma_overline_b}. 
Since $\a+m+\overline\a>-1$ the integral in the right hand side of \eqref{f_l22-3} converges as $\e$ tends to $0$. 

Therefore, by subtracting the negative power terms of $\e$ and the log term if it exists, we obtain Hadamard's finite part:
\begin{eqnarray}
\displaystyle \Pf \int_{B_{\e_0}(x)\cap W_i'}{|x-y|}^\a dy
&=&\displaystyle \lim_{\e\to0^+}\left(\int_\e^{\e_0}t^{\a+m-1}\,\overline\varphi_{x,i}(t)\,dt+\sum_{j=0}^{\overline\a}\frac{\overline\varphi_{x,i}^{(j)}(0)}{(\a+m+j)j!}\,{\e}^{\a+m+j}\right) \nonumber \\
&=&\displaystyle \sum_{j=0}^{\overline\a}\frac{\overline\varphi_{x,i}^{(j)}(0)}{(\a+m+j)j!}\,{\e_0}^{\a+m+j} \label{eq_Pf}\\
&&\displaystyle +\int_0^{\e_0}t^{\a+m-1}\left(\overline\varphi_{x,i}(t)-\sum_{j=0}^{\overline\a}\frac{\overline\varphi_{x,i}^{(j)}(0)}{j!}\,t^j\right)dt. \label{eq_Pf2}
\end{eqnarray}
Since $\big|\overline{\varphi}_{x,i}^{\,(j)}(0)\big|\le \overline b$ $(1\le j\le \overline\a)$ 
by Lemma \ref{lemma_overline_b}, \eqref{f25}, \eqref{eq_Pf} and \eqref{eq_Pf2} imply 
\[
\left|\Pf \int_{B_{\e_0}(x)\cap W_i'}{|x-y|}^\a dy \right|
\le
\frac{\sigma_{m-1}}{|\a+m|}\,{\e_0}^{\a+m}+
\sum_{j=1}^{\overline\a+1}\frac{\overline b}{|\a+m+j|\,j!}\,{\e_0}^{\a+m+j},
\]
which completes the proof. 
\end{proof}

\section{Self-repulsiveness of energies}

Let us first explain $C^q$-topology on the space of $m$-dimensional closed immersed submanifolds of class $C^k$ in $\RR^n$ ($m\le n, 0\le q\le k$) after Budney's exposition \cite{Ryan} based on Hirsch's book \cite{H}, pages 34 and 35. \\ \indent
Let $M$ and $N$ be $C^k$-manifolds. Let $C^k(M,N)$ be the set of $C^k$-maps from $M$ to $N$. 
Let $f\in C^k(M,N)$, $(\varphi,U)$ and $(\psi,V)$ be local charts of $M$ and $N$ respectively, $K\subset U$ be a compact set such that $f(K)\subset V$ and $\e>0$. 
Let $\mathcal{N}^k_q(f, (\varphi,U), (\psi,V), K, \e)$ be the set of $g\in C^k(M,N)$ that satisfy $g(K)\subset V$ and 
\[
\left|\partial^\mu(\psi g\varphi^{-1})(x)-\partial^\mu(\psi f\varphi^{-1})(x)
\right|<\e
\]
for all $x\in\varphi(K)$ and multi-index $\mu$ with $|\mu|\le q$. 
The {\it compact-open $C^q$-topology} or {\it weak topology} of $C^k(M,N)$ is the topology generated by the subbase $\{\mathcal{N}^k_q(f, (\varphi,U), (\psi,V), K, \e)\}$. \\ \noindent
Fix an $m$-dimensional closed manifold $M$. 
Let ${\rm{Diff}}^k_q(M)$ be the set of $C^k$-diffeomorphisms of $M$ and ${\rm{Imm}}^k_q(M,\RR^n)$ be the set of immersions from $M$ to $\RR^n$, both endowed with the compact-open $C^q$-topology (i.e. relative topology as the subspaces of $C^k_q(M,M)$ and $C^k_q(M,\RR^n)$ respectively). 
The group ${\rm{Diff}}^k_q(M)$ acts on ${\rm{Imm}}^k_q(M,\RR^n)$ by composition. 
The quotient space 
\[
{\rm{Imm}}^k_q(M,\RR^n)/{\rm{Diff}}^k_q(M) 
\]
is the space of submanifolds in $\RR^n$ that are the images of immersions from $M$. 
Now we take the disjoint union 
\[
\bigsqcup_{M} \>{\rm{Imm}}^k_q(M,\RR^n)/{\rm{Diff}}^k_q(M),
\]
where we take the union over all the diffeomorphism types of $m$-dimensional closed manifolds $M$.

We say that a functional on the space of $C^k$ closed submanifolds of $\RR^n$ is \jb{{\it self-repulsive}} with respect to $C^q$-topology if the value of the functional blows up as an \jb{embedded} submanifold approaches with respect to $C^q$-topology to an immersed submanifold with double points. 

\subsection{Regularized Riesz energy}
\subsubsection{Self-repulsiveness}
%
The regularized $\a$-energy $E_\a$ restricted to $\mathcal{M}(\overline\alpha+3,\e_0,b,V)$, where $\overline\a=\max\{\lfloor -\a \rfloor-m,0\}$ as in Lemma \ref{lemma_overline_b}, is self-repulsive with respect to $C^0$-topology if $\a\le-2m$ in the following sense. 

\begin{theorem}\label{theorem1}
Let $\e_0,b,V>0$. 
If $\a\le-2m$ then for any positive number $C$ there is a positive number $\delta$ such that if an embedded closed submanifold $M\in\mathcal{M}(\overline\alpha+3,\e_0,b,V)$ satisfies the following conditions then $E_\a(M)>C$. 

\begin{enumerate}
\item There are an immersed submanifold $M'\in\mathcal{M}(\overline\alpha+3,\e_0,b,V)$ with a double point $x'$ and a $C^{\overline\a+3}$-immersion $f\colon M\to M'=f(M)$ such that $|f(x)-x|<\delta$ for any $x$ in $M$. 
\item $f^{-1}\big(B_{\e_0}(x')\cap M'\,\big)$ is a disjoint union of $W_1,\dots, W_l$ $(l\ge2)$, each $W_i$ is $C^{\overline\a+3}$-diffeomorphic to an $m$-ball, and that $f|_{W_i}\colon W_i\to M'$ is a $C^{\overline\a+3}$-embedding. 
\end{enumerate}

\end{theorem}

\begin{proof}
Lemma \ref{lemma_overline_b} implies that one can find a positive number $\e_1$ with $\e_1\le\e_0$ such that if $M'\in\mathcal{M}(\overline\alpha+3,\e_0,b,V)$ then for any $y'\in M'$ and for any $\rho$ with $0<\rho\le\e_1$ we have
\[
\mbox{Vol}\left((B_\rho(y')\setminus B_{\rho/2}(y'))\cap W_i'\right)
\ge 0.9\left(1-\left(\frac12\right)^m\right)\rho^m\,\mbox{Vol}(B^m) \quack 
\mbox{for all $i=1,\dots,l$},
\]
where $W_i'$ is as in Definition \ref{def_M(e,b,V)} and $B^m$ is a unit $m$-ball. 
Put $c_1=0.9\left(1-(1/2)^m\right)\mbox{Vol}(B^m)$ so that the right hand side above is given by $c_1\rho^m$. 

Let $W_0(x, \e)$ denote the connected component of $B_{\e}(x)\cap M$ that contains a point $x$ of $M$. 

Assume $M\in\mathcal{M}(\overline\alpha+3,\e_0,b,V)$ satisfies the conditions (1) and (2) of the theorem. 
Put $x_i={(f|_{W_i})}^{-1}(x')$ $\>(i=1,2)$. 
Since the curvature is bounded there is a positive number $\e_2=\e_2(\a,\e_0,b)$ such that $B_{r}(x_i)\cap W_i$ coincides with $W_0(x_i, r)$ if $0<r\le 2\e_2$ for each $i$. 

Put $\e_3=\min\{\e_1,\e_2\,,1\}$. Consider a pair of ``annuli'' 
\[
A_1(\rho)=(B_\rho(x_1)\setminus B_{\rho/2}(x_1))\cap W_1, \>\>
A_2(\rho)=(B_\rho(x_2)\setminus B_{\rho/2}(x_2))\cap W_2 \quack (0<\rho\le\e_3).
\]
Since $|x_1-x_2|<2\delta$, if $x\in A_1(\rho)$ and $y\in A_2(\rho)$ then $|x-y|<2\rho+2\delta$. 
Therefore the interaction energy between $A_1(\rho)$ and $A_2(\rho)$ can be estimated from below by 
\[
\int_{A_1(\rho)}\int_{A_2(\rho)}{|x-y|}^\a\,dxdy 
\ge {(2\rho+2\delta)}^\a\big(c_1\rho^m\big)^2
=2^\a c_1^2\,\rho^{2m+\a}\left(1+\frac\delta\rho\right)^\a.
\]
Let the right hand side above be denoted by $\lambda_\a(\delta,\rho)$. 
It is a decreasing function of $\delta$ if we fix $\rho$. 

Since $W_0(x_i,\e_3)=B_{\e_3}(x_i)\cap W_i$ $(i=1,2)$ includes the disjoint union of the interiors of $A_i(2^{-j}\e_3)$ $(j\ge0)$, which we denote by $\mbox{int}A_i(2^{-j}\e_3)$, we have 
\[
W_0(x_1,\e_3)\times W_0(x_2,\e_3) \supset \bigsqcup_{j\ge0} \mbox{int}A_1(2^{-j}\e_3)\times \mbox{int}A_2(2^{-j}\e_3).
\]
Remark that if $x\in W_0(x_1,\e_3)$ then $W_0(x,\e_3)$ and $W_0(x_2,\e_3)$ are disjoint because $W_0(x,\e_3)\subset W_0(x_1, 2\e_2)\subset W_1$, $W_0(x_2,\e_3)\subset W_0(x_2, 2\e_2)\subset W_2$, and $W_1$ and $W_2$ are disjoint by the assumption (2). 
Therefore, 
\begin{eqnarray}
E_\a(M)&=&\displaystyle \int_M\left(\int_{M\setminus W_0(x,\e_3)}{|x-y|}^\a dy+\Pf\int_{W_0(x,\e_3)} {|x-y|}^\a dy\right)dx \nonumber \\
&\ge&\displaystyle \int_{W_0(x_1,\e_3)}\left(\int_{W_0(x_2,\e_3)}{|x-y|}^\a dy\right)dx +\int_M\left(-b^0(\a,\e_3,b)\right)dx \nonumber
\\
&\ge&\displaystyle  \sum_{j=0}^\infty \lambda_\a\left(\delta,\frac{\e_3}{2^j}\right) -b^0(\a,\e_3,b)V, \label{f_lbv}
\end{eqnarray}
where $b^0(\a,\e_3,b)$ is the constant given in Lemma \ref{lemma1}. 
Put 
\[
l_0=\left\lceil \frac{C+b^0(\a,\e_3,b)V}{2^{2\a}c_1^2} \right\rceil \quack \mbox{ and } \quack \delta_0=\frac{\e_3}{2^{l_0}},
\]
then as $\e_3\le 1$ and $2m+\a\le0$ we have
\[\begin{array}{rcl}
\displaystyle 
\sum_{j=0}^{l_0} \lambda_\a\left(\delta_0,\frac{\e_3}{2^j}\right) 
&\ge&\displaystyle \sum_{j=0}^{l_0} \lambda_\a\left(\frac{\e_3}{2^j},\frac{\e_3}{2^j}\right) 
=2^{2\a}c_1^2 \sum_{j=0}^{l_0} {\left(\frac{\e_3}{2^j}\right)}^{2m+\a} 
> 2^{2\a}c_1^2 l_0\ge C+b^0(\a,\e_3,b)V, 
\end{array}
\]
which, together with \eqref{f_lbv}, implies $E_\a(M)>C$. 
\end{proof}

We remark that in the case of knots we did not need any bound on the curvature to show the self-repulsiveness (\cite{O1}). The author does not know if the self-repulsiveness can be shown without the condition (4) of $\mathcal{M}(\overline\alpha+3,\e_0,b,V)$ in Definition \ref{def_M(e,b,V)}.

\subsubsection{Behavior of the energy of surfaces with double points}
In the case of knots, $E_\a(K')$ is finite even if $K'$ has transversal self-intersection when $\a>-2$ \cite{O94}. 
Let us study in the case of surfaces whether $E_\a(M')$ is finite or not if $M'$ has a double point when $\a>-4$ according to the type of the double point. 

\medskip
(1) Suppose $n\ge4$ and $M'$ has an orthogonal self-intersection at a point $x'$. 
Then for a sufficientlly small positive number $\rho$, the order of contribution of $M\cap B_\rho(x')$ to the energy can be estimated by 
\[\begin{array}{rcl}
\displaystyle \int_0^\rho\int_0^\rho {\left(t^2+s^2\right)}^{\a/2} {(2\pi)}^2ts\,dsdt
&=&\displaystyle 4\pi^2 \,\rho^{\a+4}\int_0^1\int_0^1 {\left(t^2+s^2\right)}^{\a/2} ts\,dsdt \\[4mm]
&=&\displaystyle \left\{
\begin{array}{ll}
\displaystyle  8\pi^2 \,\rho^{\a+4}\frac{{\sqrt2}^{\,\a+2}-1}{(\a+2)(\a+4)} & \quack\mbox{ if }\a>-4, \\
\displaystyle \infty & \quack \mbox{ if }\a\le-4.
\end{array}
\right.
\end{array}
\]

Therefore, when the dimension $n$ of the ambient space is greater than or equal to $4$, $E_\a$ is not self-repulsive if $\a>-4$. 

\medskip
(2) Suppose $M$ has a tangential double point. For example, if a neighbourhood of the tangent point is isometric to that of the union of a unit sphere and a tangent plane (the codimension does not have to be $1$ here), the contribution of a small neighbourhood of the tangent point to the energy can be estimated, up to multiplication and summation by constants, by 
\[
\begin{array}{rcl}
\displaystyle \int_0^{R}\left(\int_0^{R}{\left(\left(\frac\rho2\right)^4+r^2\right)}^{\alpha/2}r\,dr\right)\rho\,d\rho
&\sim&\displaystyle \int_0^{R}\frac{\rho^{2\a+5}}{2^{\a+2}(\a+2)}\,d\rho+O(1), \quad 0<R\ll 1, 
\end{array}
\]
which is finite if and only if $\a>-3$, where we used a convention $\xi^0/0=\log\xi$ $(\xi>0)$ as before. 

\medskip
(3) Let us consider a degenerate case. Suppose $M$ has a cone singularity, say, of the type $x^2+y^2-z^2=0$. 
Then the interaction energy between the upper and lower cones can be estimated, up to multiplication by a constant, by 
\[
\int_0^1\left(\int_0^1 {(r+\rho)}^\a \rho\,d\rho\right)r\,dr
=\int_0^1\frac{r^{\a+3}}{(\a+1)(\a+2)}\,dr+O(1),
\]
which is finite if and only if $\a>-4$. 

\medskip
Assume $\a\le-4$. 
Suppose $\{M_j\}$ is a sequence of embedded surfaces that degenerate to a singular surface with a double point with respect to $C^0$-topology. 
Let $x_j$ and $y_j$ be points in $M_j$ such that the intrinsic (geodesical) distance between them in $M_j$ is bounded below by a positive constant and that $|x_j-y_j|$ tends to $+0$ as $j$ goes to $\infty$. 
Assume that $x_j$ and $y_j$ are at the tip of conical protrusions, say $C^1_j$ and $C^2_j$, whose ``hights'' are bounded below by a positive constant. 
The interaction energy between $C^1_j$ and $C^2_j$ (the contribution of $C^1_j\times C^2_j$ to $E_\a(M_j)$) can be bounded if the ``radii'' of $C_j^i$ $(i=1,2)$ go to $+0$ according as $|x_j-y_j|$ goes to $+0$, but then we conjecture that the energies of $C_j^i$ $(i=1,2)$ blow up by the argument for the energy $E_{AS}$ of a cylinder by Auckly and Sadun \cite{AS}. 

Thus we are lead to the following conjecture:

\begin{conjecture}\rm
The regularized $\a$-energy $E_\a$ for smooth closed surfaces 
is self-repulsive with respect to $C^0$-topology if and only inf $\a\le-4$, while $E_\a$ restricted to the space of surfaces in $\RR^3$ with curvature bounded by a constant $b$ $(b>0)$ is self-repulsive if and only inf $\a\le-3$. 
\end{conjecture}

\subsection{Auckly-Sadun's regularized surface energy}
Let us first introduce some of the results of Auckly and Sadun \cite{AS}. 
Let $M$ be a $2$-dimensional closed submanifold of $\RR^n$ $(n\ge3)$ and $h$ be the second fundamental form of $M$, i.e., $h(X,Y)={(\nabla_X Y)}^\perp$ $(X,Y\in T_xM)$, where $\nabla$ is the Levi-Civita connection. 
Let $x$ be a point on $M$ and $h_{ij}$ $(i,j=1,2)$ be given by $h_{ij}=h(\vect e_i, \vect e_j)$, where $\{\vect e_1, \vect e_2\}$ is an orthonormal basis of $T_xM$. 
Put 
\[
\Delta(x)={|h_{11}-h_{22}|}^2+4{|h_{12}|}^2 \quad \mbox{ and } \quad
K(x)=\langle h_{11}, h_{22}\rangle-{|h_{12}|}^2.
\]
When $M$ is a surface in $\RR^3$, $\Delta(x)=(\kappa_1-\kappa_2)^2$ and $K(x)=\kappa_1\kappa_2$ is the Gauss curvature, where $\kappa_1$ and $\kappa_2$ are principal curvatures of $M$ at $x$. 

From a viewpoint of differential geometry, if we put $H$ to be the mean curvature vector $H=\mbox{\rm trace}\, h=h_{11}+h_{22}$ and $\Vert h\Vert$ to be the Hilbert-Schmidt norm, ${\Vert h\Vert}^2=\sum_{i,j}{\vert h_{ij}\vert}^2$, $\Delta(x)$ and $K(x)$ are given by 
\[
\Delta(x)=2{\Vert h\Vert}^2-{\vert H\vert}^2  \quad \mbox{ and } \quad 
K(x)=\frac12\Big({\vert H\vert}^2 - {\Vert h\Vert}^2\Big).
\] 
When $M$ is a surface in $\RR^3$, $H=(\kappa_1+\kappa_2)\vect n$, where $\vect n$ is a unit normal vector to $M$, and ${\Vert h\Vert}^2={\kappa_1}^2+{\kappa_2}^2$.

Then (2.9) of \cite{AS} implies 
\begin{equation}\label{AS-Pf}
%
\Pf\int_M \frac{dy}{|x-y|^4}=\lim_{\e\to0}\left(\int_{M, |x-y|\ge\e}\frac{dy}{|x-y|^4}-\frac\pi{\e^2}+\frac{\pi\Delta(x)}{8}\log\e\right). 
\end{equation}
Auckly and Sadun gave a one-parameter family of M\"obius invariant functionals $\{E_{AS}^{\,s}\}_{s\in\RR}$ by 
\[
\begin{array}{rcl}
\displaystyle 
E_{AS}^{\,s}(M)&=&\displaystyle \int_M\left(\lim_{\e\to0}\left(\int_{M, |x-y|\ge\e}\frac{dy}{|x-y|^4}-\frac\pi{\e^2}+\frac{\pi\Delta(x)}{16}\log\Delta(x)\e^2\right)+\frac\pi4K(x)\right)\,dx \\[4mm]
&&\displaystyle +s\int_M\Delta(x)\,dx
\end{array}
\]
((2.11)-(2.13) of \cite{AS}), which, together with \eqref{def_of_E_alpha} and \eqref{AS-Pf}, implies 
\begin{eqnarray}
E_{AS}^{\,s}(M)
&=&\displaystyle E_{-4}(M)
+\frac{\pi}{16}\int_M\Delta(x)\log \Delta(x)\,dx+\frac{\pi^2}2\chi(M)+s\int_M\Delta(x)\,dx, 
\label{E_AS_E4}
\end{eqnarray}
where $\chi(M)$ is the Euler characteristic. 
We remark that the second term of \eqref{E_AS_E4} was added to make $E_{AS}^{\,s}$ M\"obius invariant, and that $E_{-4}(M)$ is not M\"obius invariant as was pointed out in \cite{OS2}. 

Now the self-repulsiveness in the sense of Theorem \ref{theorem1} of Auckly-Sadun's surface energy follows from that of the regularized Riesz energy $E_{-4}$.

\subsection{Kusner-Sullivan's $\vect{(1-\cos\theta)^m}$ energy}
%
Let $M$ be an $m$-dimensional oriented closed submanifold of class $C^2$ in $\RR^n$. 
For a pair of points of $M$, $x$ and $y$, we define a M\"obius invariant angle $\theta=\theta_M(x,y)$ as follows. 

Let $\Sigma_x(y)$ be an $m$-sphere which is tangent to $M$ at $x$ that passes through $y$. Remark that $\Sigma_x(y)$ is uniquely determined. 
Suppose $\Sigma_x(y)$ is endowed with the orientation that coincides with that of $M$ at point $x$. 
Put $\Pi_y=T_yM$ and $\Pi_x=T_y\Sigma_x(y)$. 
Let $u_1,\dots, u_m$ and $v_1,\dots, v_m$ be ordered orthonormal bases of $\Pi_x$ and $\Pi_y$ that give positive orientations of $\Pi_x$ and $\Pi_y$ respectively. 
Let $\theta_M(x,y)$ be the angle between $u_1\wedge \dots \wedge u_m$ and $v_1\wedge \dots \wedge v_m$ in the Grassmannian, i.e. 
\begin{equation}\label{combined_angle}
\cos\theta_M(x,y)=\langle u_1\wedge \dots \wedge u_m, v_1\wedge \dots \wedge v_m\rangle 
=\det(\langle u_i, v_j\rangle ).  
\end{equation}
Note that $\theta_M(x,y))$ is symmetric in $x$ and $y$. 
The angle $\theta_M(x,y)$ is called {\it combined angle} in \cite{KS}. 
Now {\it Kusner-Sullivan's energy} \cite{KS} is given by 
\[
E_{KS}(M)=\iint_{M\times M}\frac{{(1-\cos\theta_M(x,y))}^m}{{|x-y|}^{2m}}\,dxdy.
\]
It is a natural generalization of the {cosine formula} of the energy for knots by Doyle and Schramm (\cite{AS,KS}). 
We show the self-repulsiveness under the assumption of bounded geometry.

\begin{lemma}\label{eigenvalues}
\begin{enumerate}
\item The eigenvalues of a matrix $A=(\a_i\a_j)_{1\le i,j\le m}$, where $\a_1,\dots,\a_m$ are real numbers, are given by $\a_1^2+\dots+\a_m^2,0,\dots,0$.
\item Suppose $w_1,\dots,w_m$ is an ordered basis of $\Pi_y$ that gives the positive orientation. Then 
\[
\det(\langle u_i, v_j\rangle)=\frac{\det(\langle u_i, w_j\rangle)}{\sqrt{\det(\langle w_i, w_j\rangle)}}. 
\]
\end{enumerate}
\end{lemma}
\begin{proof}
The vectors in the Euclidean spaces are considered as column vectors here. \\ \indent
(1) Put $\vect \a={}^t(\a_1,\dots,\a_m)$. Assume $\vect \a\ne\vect0$ as otherwise it is trivial. 
Let $\vect x={}^t(x_1,\dots,x_m)\in \RR^m$. Since
\[
A\vect x=\left(\begin{array}{c}
\a_1\\
\vdots \\
\a_m
\end{array}
\right)
\left(\a_1 \dots \a_m\right)
\left(\begin{array}{c}
x_1\\
\vdots \\
x_m
\end{array}
\right)
=\biggl(\sum_k\a_kx_k
\biggr)
\left(\begin{array}{c}
\a_1\\
\vdots \\
\a_m
\end{array}
\right),
\]
if $\vect x$ is an eigenvector, then either $\vect x$ is a nonzero constant times $\vect \a$ with eigenvalue ${|\vect \a|}^2$ or $\vect x$ is orthogonal to $\vect \a$ with eigenvalue $0$, which completes the proof. \\ 
\indent
(2) We have $(v_1,\dots,v_m)=(w_1,\dots,w_m)A$ for some non-singular matrix $A$ with positive determinant. Then, $v_1\wedge \dots \wedge v_m=|A|w_1\wedge \dots \wedge w_m$ and $|A|^2 \det(\langle w_i, w_j\rangle)=1$. 
\end{proof}
\begin{lemma}\label{conformal_angle_estimate} 
Suppose $M$ is locally expressed as a graph of a function 
\[
h\colon U\to {(T_{x}M)}^\perp\cong\RR^{n-m}, \quack (\vect 0\in U\subset T_{x}M\cong\RR^{m}). 
\]
Then $1-\cos\theta_M(x,y)$ which appears in the numerator of the integrand of $E_{KS}$ is of the order of $|x-y|^2$ and 
$|1-\cos\theta_M(x,y)|/|x-y|^2$, where $y=(\eta, h(\eta))$ for some $\eta\in U$, can be estimated by 
$\sup_{\eta', \mu} |\partial^\mu h(\eta')|$, where $\eta'\in U$ and $\mu$ is a multi-index with $|\mu|= 2$. 
\end{lemma}
\begin{proof}
We may assume that $x=\vect 0$, $T_xM=\RR^m$ and a neighbourhood of $x=\vect 0$ of $M$ is given as a graph of $h\colon U\to\RR^{n-m}$ $(\vect 0\in U\subset \RR^m)$. 
Let $e_i$ be the canonical unit vector. 
Suppose $y={}^t(\eta,h(\eta))$. 
Let $v$ be a unit vector in the direction of $y-x$. 
Suppose the orientation of $M$ agrees with the ordered basis $e_1,\dots,e_m$ at $x=\vect 0$. 
Let $\hat e_i$ be a vector symmetric to $e_i$ with respect to the line through $x$ and $y$: 
\[
\hat e_i=2\langle e_i,v\rangle v-e_i
=\frac{2 \eta_i}{|\eta|^2+|h(\eta)|^2}\left(
\begin{array}{c}
\eta \\
h(\eta)
\end{array}
\right)
-\left(
\begin{array}{c}
e_i \\
0
\end{array}
\right)
\qquad 1\le i\le m.
\]
Then $\hat e_1,\dots, \hat e_m$ is $(-1)^{m-1}$ times the ordered orthonormal basis of $\Pi_x=T_y\Sigma_x(y)$ that gives the positive orientation. 
On the other hand, if we put $w_j={}^t(e_j,\partial_jh(\eta))$ then $w_1,\dots,w_m$ is an ordered basis that gives the positive orientation of $\Pi_y$. 
We have 
\begin{eqnarray}
\displaystyle \langle \hat e_i, w_j\rangle
&=&\displaystyle \frac{2 \eta_i\big(\eta_j+\big\langle\partial_jh(\eta), h(\eta)\big\rangle\big)}{|\eta|^2+|h(\eta)|^2}-\delta_{ij}
=\frac{2 \eta_i\eta_j}{|\eta|^2}-\delta_{ij}+O(|\eta|^2), \label{ehatw} \\[2mm]
\displaystyle \langle \hat w_i, w_j\rangle&=&\displaystyle \big\langle\partial_ih(\eta), \partial_jh(\eta)\big\rangle+\delta_{ij}
=\delta_{ij}+O(|\eta|^2). \label{ww}
\end{eqnarray}
Since Lemma \ref{eigenvalues} (1) implies that the eigenvalues of a matrix $\big({2 \eta_i\eta_j}{|\eta|^{-2}}-\delta_{ij}\big)$ are $1,-1,\dots,-1$, \eqref{combined_angle} and Lemma \ref{eigenvalues} (2) imply 
\[
\begin{array}{rcl}
\cos\theta_M(\vect 0,(\eta,h(\eta)))&=&\displaystyle (-1)^{m-1}\frac{\det(\langle \hat e_i, w_j\rangle)}{\sqrt{\det(\langle w_i, w_j\rangle)}} 
=1-O(|\eta|^2). 
\end{array}
\]
Since $\partial_kh(0)=0$ for any $k$ $(1\le k\le m)$, both $|\partial_jh(\eta)|/|\eta|$ and $|h(\eta)|/|\eta|^2$ can be estimated by $|\partial ^\mu h|$ with $|\mu|=2$, which implies that both 
\[
\frac1{|\eta|^2}\left[
\langle \hat e_i, w_j\rangle
-\left( \frac{2 \eta_i\eta_j}{|\eta|^2}-\delta_{ij}\right)\right]
\quad\mbox{ and }\quad
\frac1{|\eta|^2}\left(
\langle \hat w_i, w_j\rangle
-\delta_{ij}\right)
\]
can be estimated by $|\partial ^\mu h|$ with $|\mu|=2$. 
Therefore 
\[
\frac1{|\eta|^2}\left(1-\cos\theta_M(\vect 0,(\eta,h(\eta)))\right)
\]
can be estimated by $|\partial ^\mu h|$ with $|\mu|=2$, which completes the proof. 
\end{proof}
Note that $\theta_M(x,y)$ can be controlled by the $C^1$-topology on the off-diagonal set when $|x-y|$ is bounded below by a positive constant. 
\begin{corollary}
The combined angle $\theta_M(x,y)$ is of the order of $|x-y|$. 
\end{corollary}
\begin{remark}\rm 
When $m=1$, i.e. in the case of knots $K$, $\theta_K(x,y)$ is of the order of $|x-y|^2$ (cf. \cite{LO} Lemma 5.3). It does not hold when $m\ge2$. For example, consider a graph of $h(x,y)=xy$, then $\theta_M(\vect 0, (t,t))\approx t$ $(0\le t\ll 1)$. 
\end{remark}
\begin{corollary}
Kusner-Sullivan's energy $E_{KS}$ is continuous with respect to the $C^2$-topology \jb{on the space of embedded submanifolds}. 
\end{corollary}
\jb{
\begin{remark}\rm 
$E_{KS}$ is not continuous with respect to the $C^2$-topology on the space of immersed submanifolds. For example, the value of $E_{KS}$ of a double covered sphere $S^m\sqcup S^m\to S^m$ is equal to $0$ since the combined angle $\cos_M(x,y)$ is always $0$, whereas the value of $E_{KS}$ of a pair of concentric spheres with radii $1$ and $1+\e$ $(\e\ne0)$ 
tends to $+\infty$ as $\e$ approaches $0$. 
\end{remark}
}
\begin{lemma}\label{conf_angle}
Suppose $T_xM$ and $T_yM$ are parallel and with the same orientation. 
Let $v$ be a unit vector in the direction of $y-x$. 
Then the combined angle $\theta_M(x,y)$ satisfies 
\[
\cos\theta_M(x,y)=2{|\mbox{\rm pr}(v)|}^2-1,
\]
where $\mbox{\rm pr}:\RR^n\to T_xM$ is the orthogonal projection. 
\end{lemma}
\begin{proof} 
This can be proved in the same way as the previous lemma. 
Since $T_xM$ and $T_yM$ are parallel with the same orientation we can choose the same ordered orthonormal basis $u_1,\dots,u_m$ for both $T_xM$ and $T_yM$. 
Let $\hat u_i$ be a vector symmetric to $u_i$ with respect to the line through $x$ and $y$: $\hat u_i=2\langle u_i,v\rangle v-u_i,$ 
where $v=(y-x)/|y-x|$. 
Then $\hat u_1,\dots, \hat u_m$ is $(-1)^{m-1}$ times the ordered orthonormal basis of $\Pi_x=T_y\Sigma_x(y)$ that gives the positive orientation, which implies 
\[
\begin{array}{rcl}
\cos\theta_M(x,y)
&=&\displaystyle (-1)^{m-1}\det(2\langle u_i,v\rangle\langle u_j,v\rangle-\langle u_i,u_j\rangle). \\[2mm]
%
\end{array}
\]
Lemma \ref{eigenvalues} (1) implies that the eigenvalues of $2(\langle u_i,v\rangle\langle u_j,v\rangle)_{1\le i,j\le m}-I$ are 
$2\sum_i\langle u_i,v\rangle^2-1,-1,\dots,-1$, which completes the proof. 
\end{proof}
\begin{theorem}\label{theorem2}
Kusner-Sullivan's energy $E_{KS}$, 
restricted to $\mathcal{M}(2,\e_0,b,+\infty)$ (cf. Definition \ref{def_M(e,b,V)}), 
is self-repulsive with respect to the $C^2$-topology. 
\end{theorem}

\begin{proof} 
\jb{Suppose a sequence of embedded compact submanifolds $M_k$ of class $C^2$ $(k\in\NN)$ converges to an immersed submanifold $M_{\jb{\infty}}$ in $\mathcal{M}(2,\e_0,b,+\infty)$ with a double point $x'$ with respect to the $C^2$-topology. 
}
%
Suppose 
we have $B_{\e_0}(x')\cap M_{\jb{\infty}}=W_{\jb{\infty,1}}\cup W_{\jb{\infty,2}}$ with $W_{\jb{\infty,i}}$ being diffeomorphic to $B^m$. 
Let $\Pi_i$ be $T_{x'}W_{\jb{\infty,i}}$ $(i=1,2)$. 
Assume that each $W_{\jb{\infty,i}}$ can be expressed as graph of a function $h_{\jb{\infty,i}}$ from a subset $U_{\jb{\infty,i}}$ of $\Pi_i$ to ${\jb{\Pi_i}}^\perp\cong\RR^{n-m}$. 
\jb{
Since $M_k$ converges to $M_\infty$, by taking a subsequence if necessary, we may assume that for each $k$ 
$B_{\e_0}(x')\cap M_k=W_{k,1}\cup W_{k,2}$ with $W_{k,i}$ being diffeomorphic to $B^m$, and each $W_{k,i}$ can be expressed as graph of a function $h_{k,i}$ from a subset $U_{k,i}$ of $\Pi_i$ to ${\Pi_i}^\perp$. 
}
We abuse the notation to express a point $(\eta,h_{\jb{k,i}}(\eta))$ in $\RR^n$ $(\eta\in U_{\jb{k,i}})$ by $h_{\jb{k,i}}(\eta)$ in what follows for the sake of simplicity \jb{$(k\in\NN\cup\{\infty\})$}. 
We use the coordinates of $\Pi_i$ so that $x'$ is the origin in what follows. 
Let $U_i$ be the interior of $\cap_{k=1}^\infty U_{k,i}$ $(i=1,2)$. 
\jb{By taking a subsequence of $\{M_k\}$ if necessary, and by taking $\e_0$ sufficiently small, we may assume the following. 
\begin{itemize}
\item For $i=1,2$ $U_i$ contains the origin. 
\item For any $k$ and for any multi-index $\mu$ with $|\mu|\le2$ there hold 
\begin{equation}\label{closeness}
\begin{array}{rl}
\displaystyle |\partial^\mu h_{k,1}(p)-\partial^\mu h_{\infty,1}(p)|<2^{-k} &\qquad \forall p\in U_1, \\[1mm]
\displaystyle |\partial^\mu h_{k,2}(q)-\partial^\mu h_{\infty,2}(q)|<2^{-k} &\qquad \forall q\in U_2. 
\end{array}
\end{equation}
\end{itemize}
We remark that \eqref{closeness} implies that $W_{\jb{k,i}}$ is sufficiently close to $U_i$ with respect to the $C^1$-topology for a sufficiently large $k$ $(i=1,2)$. 
}

We show that for any positive number $C$ there are natural numbers \jb{$k_0$ and} $l$, and sets of mutually disjoint subsets 
\jb{$\{V^1_{k,j}\}_{1\le j\le l}$ of $U_1$ and $\{V^2_{k,j}\}_{1\le j\le l}$ of $U_2$ for each $k$} 
such that \jb{if $k\ge k_0$ then }
%
%
%
\begin{equation}\label{last_eq}
\sum_{j=1}^l\int_{h_{\jb{k,1}}(V_{\jb{k,j}}^1)}\int_{h_{\jb{k,2}}(V_{\jb{k,j}}^2)}\frac{{(1-\cos\theta_{M_{\jb{k}}}(x,y))}^m}{{|x-y|}^{2m}}\,dxdy>C,
\end{equation}
%
%
\jb{which implies that $E_{KS}(M_k)$ tends to $+\infty$ as $k$ goes to $+\infty$ since the integrand of $E_{KS}$ is non-negative. }

To satisfy the condition \jb{\eqref{last_eq}}, we look for $V^i_{\jb{k,j}}$ $(i=1,2,\,\jb{k,} j\in\NN)$ so that we can find a natural number $k_0$ and positive constants $c_2,c_3$ such that \jb{if $k\ge k_0$ then the following hold. }
\begin{enumerate}
\item[(a)] If $x\in h_{\jb{k,1}}(V_{\jb{k,1}}^1)$ and $y\in h_{\jb{k,2}}(V_{\jb{k,1}}^2)$ then $1-\cos\theta_{M_{\jb{k}}}(x,y)\ge c_2$ for any $j$, 
\item[(b)] $\displaystyle {\left(\max_{x\in h_{\jb{k,1}}(V_{\jb{k,1}}^1),\,y\in h_{\jb{k,2}}(V_{\jb{k,1}}^2)}|x-y|\right)}^{-2m}\,\mbox{Vol}\left(h_{\jb{k,1}}(V_{\jb{k,1}}^1)\right)\mbox{Vol}\left(h_{\jb{k,2}}(V_{\jb{k,1}}^2)\right)\ge c_3$ for any $j$. 
\end{enumerate}

This can be done as follows. 

\smallskip
Case 1. Suppose $W_{\jb{\infty,1}}$ and $W_{\jb{\infty,2}}$ are not tangent \jb{at $x'$}, i.e. $\Pi_1\ne \Pi_2$. 
\jb{In this case we can take $V_{k,j}^i$ independent of $k$, which we denote by $V_j^i$, as follows.} 
One can find a pair of points $p_0\in \Pi_1\setminus(\Pi_1\cap\Pi_2)$ and $q_0\in\Pi_2\setminus(\Pi_1\cap\Pi_2)$ (cf. Fig. 1) and 
a positive constant $c_2$ such that 
if we put $r_0=\min\{|p_0|, |q_0|\}/10$, $V_0^1=B_{r_0}(p_0)\cap\Pi_1$ and $V_0^2=B_{r_0}(q_0)\cap\Pi_2$ then $V_0^1\cap V_0^2=\emptyset$ and for any $p\in V_0^1$ and $q\in V_0^2$ 
there holds $1-\cos\theta_{\Pi_1\cup\Pi_2}(p,q)>2c_2$. 
Put $p_j=2^{-j}p_0$, $q_j=2^{-j}q_0$, $r_j=2^{-j}r_0$, $V_j^1=B_{r_j}(p_j)\cap\Pi_1$ and $V_j^2=B_{r_j}r(q_j)\cap\Pi_2$. 
Remark that for any $j$ and for any $p\in V_j^1$ and $q\in V_j^2$ there holds $1-\cos\theta_{\Pi_1\cup\Pi_2}(p,q)>2c_2$ 
since $\Pi_1$ and $\Pi_2$ are affine planes and the pictures are homothetic. 
One can take natural numbers \jb{$k_0$ and} $j_0$ such that for any \jb{$k$ and} $j$ with \jb{$k\ge k_0$ and} $j\ge j_0$ and for any $p\in V_j^1$ and $q\in V_j^2$ we have 
$|\theta_{\Pi_1\cup\Pi_2}(p, q)-\theta_{M_{\jb{k}}}(h_{\jb{k,1}}(p),h_{\jb{k,2}}(q))|<c_2$, which implies the condition (a) above. 
The choice of $r_0$ implies that $V_j^1$ and $V_j^2$ $j\in\NN$ are mutually disjoint. 
Finally, 
(b) follows from 
\[
\begin{array}{l}
\mbox{Vol}\left(h_{\jb{k,i}}(V_j^i)\right)\ge \mbox{Vol}\left(V_j^i\right) = 2^{-jm}r_0^m\mbox{Vol}(B^m) \quad \jb{\forall k}, \,i=1,2,\\[1mm]
|p-q|\le 2^{-j}(|p_0-q_0|+2r_0) \quad \forall p\in V_j^1, \forall q\in V_j^2, \, \forall j, \\[1mm]
|h_{\jb{k,1}}(p)-p|\le \frac12 b|p|^2\jb{+2^{-k}}=2^{-2j-1}b(|p_0|+|r_0|)^2\jb{+2^{-k}} \quad \forall p\in V_j^1, \, \forall j, \jb{\, \forall k,} \\[2mm]
|h_{\jb{k,2}}(q)-q|\le 2^{-2j-1}b(|q_0|+|r_0|)^2\jb{+2^{-k}} \quad \forall q\in V_j^2, \, \forall j, \jb{\, \forall k}.
\end{array}
\]

\begin{figure}[htbp]
\begin{center}
\begin{minipage}{.36\linewidth}
\begin{center}
\includegraphics[width=\linewidth]{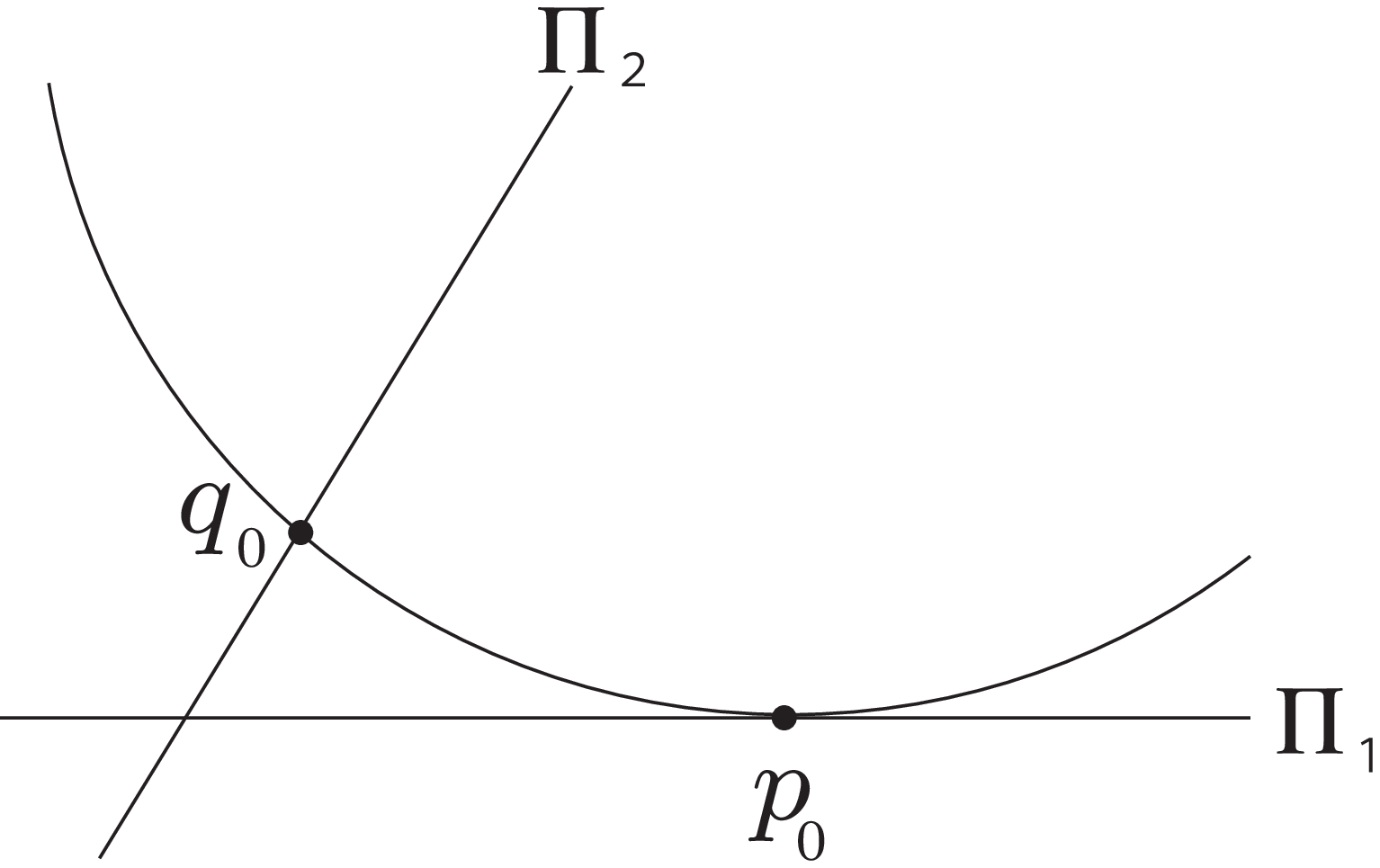}
\caption{Case 1}
\label{transverse}
\end{center}
\end{minipage}
\hskip 0.6cm
\begin{minipage}{.5\linewidth}
\begin{center}
\includegraphics[width=\linewidth]{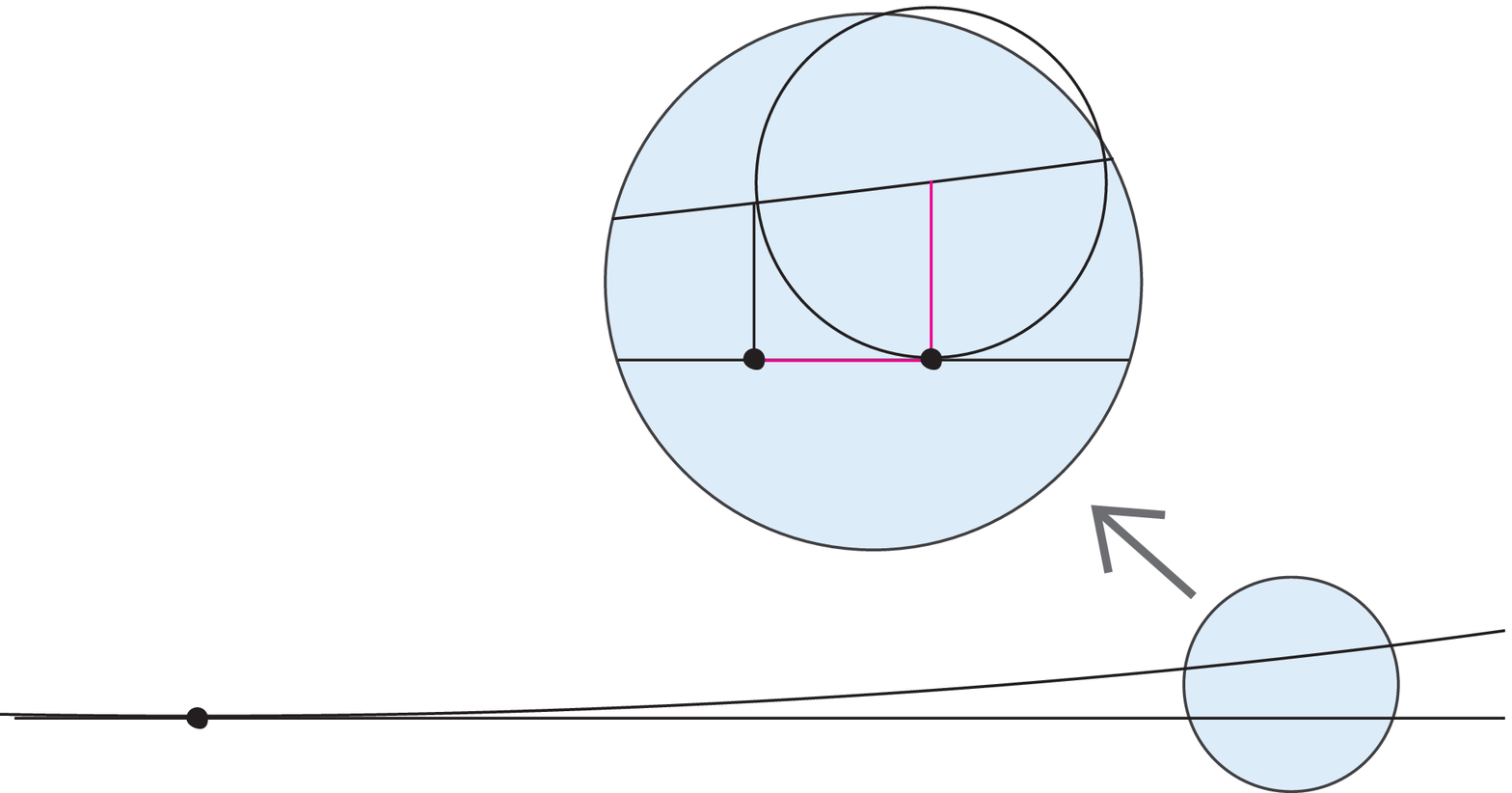}
\caption{Case 2. A picture when $h_{\jb{k,}2}\equiv0$}
\label{tangent}
\end{center}
\end{minipage}
\end{center}
\end{figure}

\smallskip
Case 2. Suppose $W_{\jb{\infty,2}}$ and $W_{\jb{\infty,2}}$ are tangent \jb{at $x'$}, i.e. $\Pi_1= \Pi_2$, which we denote by $\Pi$. 
Put $U=U_1\cap U_2\cap \stackrel\circ B_{\e_1}(\vect 0)$, where $\stackrel\circ B_{\e_1}(\vect 0)$ is an open ball with center the origin and the radius $\e_1$ $(0<\e_1\le\e_0)$. 
%
By choosing $\e_1$ sufficiently small \jb{and by taking a subsequence of $M_k$ if necessary}, we may assume that $T_{h_{\jb{k,i}}(\eta)}W_{\jb{k,i}}$ $(i=1,2, \eta\in \jb{U})$ are almost parallel to $\Pi$ \jb{for any $k$}. 
Let $p_0$ \jb{$(p_0\ne\vect 0)$} be \jb{a point in $U$}. 
Put for $j\in \NN$ 
\[
\begin{array}{c}
p_j=\displaystyle 2^{-j}p_0, \quad 
\delta_{\jb{k,}j}=\displaystyle |(h_{\jb{k,1}}-h_{\jb{k,2}})(p_j)|, \quad 
q_{\jb{k,}j}=\displaystyle p_j-\delta_{\jb{k,}j}\frac{p_j}{|p_j|} \quad \mbox{(cf. Fig. 2)}, \\[2mm]
V_{\jb{k,}j}^1=\displaystyle B_{\delta_{\jb{k,}j}/10}(p_j), \quad 
V_{\jb{k,}j}^2=\displaystyle B_{\delta_{\jb{k,}j}/10}(q_{\jb{k,}j}).
\end{array}
\]
We remark that $\delta_{\jb{k,}j}\le2\cdot\frac12b|p_j|^2\jb{+2\cdot 2^{-k}}=2^{-2j}b|p_0|^2\jb{+2^{1-k}}$, hence by choosing $\e_1$ sufficiently small, 
\jb{we can find a natural number $k_1$ such that if $k\ge k_1$ then} 
$V_{\jb{k,}j}^1$ $(j\in\NN)$ are mutually disjoint and $V_{\jb{k,}j}^2$ $(j\in\NN)$ are also mutually disjoint. 

If $h_{\jb{k,1}}(V_{\jb{k,}j}^1)$ and $h_{\jb{k,2}}(V_{\jb{k,}j}^2)$ are parallel to $\Pi$ with \jb{the difference in $\Pi^\perp$ being equal to} $\jb{\delta=}\delta_{\jb{k,}j}$ then \jb{by} Lemma \ref{conf_angle} \jb{we have}
\[
\jb{1-\cos\theta_{M_{\jb{k}}}(x,y)
\ge 1-\left(2\frac{{(\delta+\frac2{10}\delta)}^2}{\delta^2+{(\delta+\frac2{10}\delta)}^2}-1\right)=\frac{50}{61} \qquad \forall x\in h_{\jb{k,1}}(V_{\jb{k,}j}^1), \, \forall y\in h_{\jb{k,2}}(V_{\jb{k,}j}^2),}
\]
%
\noindent
which guarantees the existence of a positive constant $c_2$ of (a) \jb{in a general case since $\cos_{M_k}(x,y)$ is continuous with respect to the $C^1$-topology}. 
Under the same assumption, \jb{i.e. if $h_{\jb{k,1}}(V_{\jb{k,}j}^1)$ and $h_{\jb{k,2}}(V_{\jb{k,}j}^2)$ are parallel to $\Pi$,} we have $|x-y|\le(\sqrt{61}/5)\delta_{\jb{k,}j}$, which, together with $\mbox{Vol}\left(h_{\jb{k,i}}(V_{\jb{k,}j}^i)\right)\ge(\delta_{\jb{k,}j}/10)^m\mbox{Vol}(B^m)$ $(i=1,2)$, guarantees the existence of a positive constant $c_3$ of (b) \jb{in a general case because of the continuity with respect to the $C^1$-topology}. 
\end{proof}

We remark that even in the case of knots we used the bound of length and curvature to show the self-repulsiveness for technical reason (cf. \cite{LO} Theorem 5.7).

\end{document}